\documentclass{amsart}

\usepackage{amssymb,amsthm,amsmath,amsxtra}
\usepackage[all]{xy}
\usepackage{verbatim}

\numberwithin{equation}{section}

\theoremstyle{plain}
\newtheorem{thm}[equation]{Theorem}
\newtheorem{prop}[equation]{Proposition}
\newtheorem{lem}[equation]{Lemma}
\newtheorem{cor}[equation]{Corollary}

\newtheorem{defn}[equation]{Definition}
\newtheorem*{prop*}{Proposition}
\newtheorem*{thm*}{Theorem}
\newtheorem*{thma*}{Theorem A}
\newtheorem*{thmb*}{Theorem B}
\newtheorem*{thmc*}{Theorem C}

\theoremstyle{remark}

\newtheorem{exm}[equation]{Example}
\newtheorem{rmk}[equation]{Remark}

\newenvironment{enumroman}
{\begin{enumerate}}
{\end{enumerate}}

\entrymodifiers={+!!<0pt,\fontdimen22\textfont2>}

\DeclareMathOperator{\End}{End}

\DeclareMathOperator{\Hom}{Hom}

\DeclareMathOperator{\nrd}{nrd}

\DeclareMathOperator{\rk}{rk}

\DeclareMathOperator{\Spec}{Spec}

\DeclareMathOperator{\Sym}{Sym}
\DeclareMathOperator{\trd}{trd}
\DeclareMathOperator{\Ten}{Ten}
\DeclareMathOperator{\Tr}{Tr}

\newcommand{\Q}{\mathbb Q}
\newcommand{\R}{\mathbb R}
\newcommand{\Z}{\mathbb Z}

\newcommand{\tbigwedge}{\textstyle{\bigwedge}}

\newcommand{\eps}{\epsilon}

\newcommand{\fraka}{\mathfrak{a}}
\newcommand{\frakb}{\mathfrak{b}}
\newcommand{\frakc}{\mathfrak{c}}

\newcommand{\frakm}{\mathfrak{m}}
\newcommand{\frakp}{\mathfrak{p}}

\newcommand{\quat}[2]{\displaystyle{\biggl(\frac{#1}{#2}\biggr)}}

\newcommand{\la}{\langle}
\newcommand{\ra}{\rangle}

\SelectTips{eu}{12}

\begin{document}

\title[Characterizing quaternion rings]{Characterizing quaternion rings over an arbitrary base}

\author{John Voight}
\address{Department of Mathematics and Statistics, University of Vermont, 16 Colchester Ave, Burlington, VT 05401, USA}
\email{jvoight@gmail.com}
\date{\today}
\thanks{\emph{Subject classification:} 11R52, 11E20, 11E76}

\begin{abstract}
We consider the class of algebras of rank $4$ equipped with a standard involution over an arbitrary base ring.  In particular, we characterize quaternion rings, those algebras defined by the construction of the even Clifford algebra.
\end{abstract} 

\maketitle

A \emph{quaternion algebra} is a central simple algebra of dimension $4$ over a field $F$.  Generalizations of the notion of quaternion algebra to other commutative base rings $R$ have been considered by Kanzaki \cite{Kanzaki}, Hahn \cite{Hahn}, Knus \cite{KnusQHFR}, and most recently Gross and Lucianovic \cite{GrossLuc}.  In this article, we pursue these generalizations further.

Let $R$ be a commutative Noetherian ring (with $1$).  Let $B$ be an algebra over $R$, an associative ring with $1$ equipped with an embedding $R \hookrightarrow B$ of rings (mapping $1 \in R$ to $1 \in B$) whose image lies in the center of $B$; we identify $R$ with this image in $B$.  Assume further that $B$ is a finitely generated, projective $R$-module of constant rank.

A \emph{standard involution} on $B$ is an $R$-linear (anti-)involution $\overline{\phantom{x}}:B \to B$ such that $x\overline{x} \in R$ for all $x \in B$.  Given an algebra $B$ with a standard involution, we define the \emph{reduced trace} $\trd:B \to R$ by $x \mapsto x+\overline{x}$ and the \emph{reduced norm} $\nrd:B \to R$ by $x \mapsto x\overline{x}$.  Then every element $x \in B$ satisfies the polynomial $\mu(x;T)=T^2-\trd(x)T+\nrd(x)$.

A \emph{free quaternion ring} $B$ over a PID or local ring $R$ is an $R$-algebra of rank $4$ with a standard involution such that the characteristic polynomial $\chi(x;T)$ of left multiplication by $x$ on $B$ is equal to $\mu(x;T)^2=(T^2-\trd(x)T+\nrd(x))^2$.  The result of Gross and Lucianovic is as follows.

\begin{prop*}[{\cite[Proposition 4.1]{GrossLuc}}]
Let $R$ be a PID or local ring.  Then there is a bijection between the space of ternary quadratic forms over $R$ under a twisted action of $GL_3(R)$ and isomorphism classes of free quaternion rings over $R$.
\end{prop*}

In this correspondence, one associates to a ternary quadratic form $q$ the even Clifford algebra $C^0(q)$.  (Here, the usual action of $GL_3(R)$ on quadratic forms is twisted by the determinant; see Proposition \ref{GrossLucN3} for more details.)  In this article, we generalize this result and treat an arbitrary commutative base ring $R$.  

A \emph{ternary quadratic module} is a triple $(M,I,q)$ where $M$ is a projective $R$-module of rank $3$, $I$ is an invertible $R$-module (projective of rank $1$), and $q:M \to I$ is a quadratic map.  To a ternary quadratic module $(M,I,q)$, one can associate the \emph{even Clifford algebra} $C^0(M,I,q)$, an $R$-algebra of rank $4$ with standard involution.  A \emph{quaternion ring} is an $R$-algebra $B$ such that there exists a ternary quadratic module $(M,I,q)$ with $B \cong C^0(M,I,q)$.

Conversely, to an $R$-algebra $B$ of rank $4$ with a standard involution, we associate the quadratic map 
\[ \phi_B:\tbigwedge^2 (B/R) \to \tbigwedge^4 B \]
which is uniquely characterized by the formula
\[ \phi_B(x \wedge y) = 1 \wedge x \wedge y \wedge xy \]
for $x,y \in B$.  We call $\phi_B$ the \emph{canonical exterior form}; it can be found in a footnote of Gross and Lucianovic \cite{GrossLuc} and is inspired by the case of commutative quartic rings, investigated by Bhargava \cite{BhargavaQuartic}.

An \emph{exceptional ring} is an $R$-algebra $B$ (of rank $4$) with the property that there is a left ideal $M \subset B$ with $B=R \oplus M$ such that the left multiplication map $M \to \End_R(M)$ factors through a linear map $t:M \to R$, i.e.\ $xy = t(x)y$ for $x,y \in M$.  The map $x \mapsto t(x)-x$ for $x \in M$ extends to a standard involution on $B$.  

Our first main result is as follows.

\begin{thma*}
Let $B$ be an $R$-algebra of rank $4$ with a standard involution.  Then the following are equivalent.
\begin{enumroman}
\item $B$ is a quaternion ring, i.e., there exists a ternary quadratic module $(M,I,q)$ such that $B \cong C^0(M,I,q)$;
\item For all $x \in B$, the characteristic polynomial of left (or right) multiplication by $x$ on $B$ is equal to $\mu(x;T)^2$;
\item For all $x \in B$, the trace of left (or right) multiplication by $x$ on $B$ is equal to $2\trd(x)$.
\end{enumroman}
If $R$ is reduced and $2$ is a nonzerodivisor in $R$, then these are further equivalent to the following.
\begin{enumroman}
\item[(iv)] For all maximal ideals $\frakm$ of $R$, the reduction of the canonical exterior form 
\[ \phi_{B,\frakm} : \tbigwedge^2 (B/R) \otimes_R R/\frakm \to \tbigwedge^4 B \otimes R/\frakm \] 
is zero if and only if $B \otimes R/\frakm$ is commutative. 
\end{enumroman}

$B$ is an exceptional ring if and only if the canonical exterior form $\phi_B$ is identically zero.

Moreover, there exists a unique decomposition $\Spec R_Q \cup \Spec R_E \hookrightarrow \Spec R$ such that the restriction $B_{R_Q}=B \otimes_R R_Q$ is a quaternion ring, $B_{R_E}$ is an exceptional ring, and $\Spec R_Q$ and $\Spec R_E$ are the largest (closed) subschemes with these properties.  If $R$ is reduced, then the map $R \to R_Q \times R_E$ is surjective.  
\end{thma*}

For any prime $\frakp \in \Spec R_Q \cap \Spec R_E$, we have $B \otimes R/\frakp \cong (R/\frakp)[i,j,k]/(i,j,k)^2$.

Our second main result addresses the following issue: although by definition the association $(M,I,q) \mapsto C^0(M,I,q)$ yields all isomorphism classes of quaternion rings, it does not yield a bijection over a general commutative ring $R$.  However, we may recover a bijection by rigidifying the situation as follows.  A \emph{parity factorization} of an invertible $R$-module $N$ is an $R$-module isomorphism 
\[ p:P^{\otimes 2} \otimes Q \xrightarrow{\sim} N \]
where $P,Q$ are invertible $R$-modules.  Our second main result is the following (Theorem \ref{mainthmcorr}).

\begin{thmb*}
There is a bijection
\begin{equation*}
\left\{ \begin{minipage}{28ex} 
\begin{center}
Isometry classes of ternary \\
quadratic modules $(M,I,q)$ \\
over $R$
\end{center} 
\end{minipage}
\right\} \longleftrightarrow \left\{ \begin{minipage}{39ex} 
\begin{center}
Isomorphism classes of quaternion \\
rings $B$ over $R$ equipped with a parity \\
factorization $p:P^{\otimes 2} \otimes Q \xrightarrow{\sim} \tbigwedge^4 B$
\end{center} 
\end{minipage}
\right\}
\end{equation*}
which is functorial with respect to the base ring $R$.  In this bijection, the isometry class of a quadratic module $(M,I,q)$ maps to the isomorphism class of the quaternion ring $C^0(M,I,q)$ equipped with the parity factorization
\begin{equation*}
(\tbigwedge^3 M \otimes (I\spcheck)^{\otimes 2})^{\otimes 2} \otimes I \xrightarrow{\sim} \tbigwedge^4 C^0(M,I,q).
\end{equation*}
\end{thmb*}

Theorem B globalizes the result of Gross and Lucianovic (see Remark \ref{reducetogrossluc}).  This result also compares to work of Balaji \cite{Balaji}, who takes an alternative perspective.  

This article is organized as follows.  We begin (\S 1) by very briefly reviewing some preliminary material concerning algebras with standard involution and exceptional rings as well as ternary quadratic modules and Clifford algebras.   We then define the canonical exterior form (\S 2).  Using this form, we prove the equivalence in Theorem B (\S 3).  We then prove Theorem A (\S 4) and conclude by discussing some consequences.

The author would like to thank Asher Auel, Manjul Bhargava, Brian Conrad, Noam Elkies, Jordan Ellenberg, Jon Hanke, Hendrik Lenstra, and Raman Parimala for their suggestions and comments which helped to shape this research.  We are particularly indebted to Melanie Wood who made many helpful remarks and corrections.  This project was partially
supported by the National Science Foundation under Grant No.\ DMS-0901971.

\section{Involutions and exceptional rings}

Throughout this article, let $R$ be a commutative ring and let $B$ be an algebra over $R$, which as in the introduction is defined to be an associative ring with $1$ equipped with an embedding $R \hookrightarrow B$ of rings.  We assume further that $B$ is finitely generated, projective $R$-module of constant rank.  For a prime $\frakp$ of $R$, we denote by $R_\frakp$ the localization of $R$ at $\frakp$; we abbreviate $B_\frakp = B \otimes_R R_\frakp$ and for $x \in B$ we write $x_\frakp = x \otimes 1 \in B_\frakp$.

\begin{rmk}
We choose to work over base rings, but one could extend our results to an arbitrary (separated) base scheme by the usual patching arguments when one restricts to algebras which are locally free over the base.
\end{rmk}

In this section, we give a summary of the results we will use concerning rings with a standard involution \cite{Voightlowrank} and the construction of the Clifford algebra.

\begin{defn}
An \emph{involution (of the first kind)} $\overline{\phantom{x}}:B \to B$ is an $R$-linear map which satisfies:
\begin{enumroman}
\item $\overline{1}=1$, 
\item $\overline{\phantom{x}}$ is an anti-automorphism, i.e., $\overline{xy}=\overline{y}\,\overline{x}$ for all $x,y \in B$, and
\item $\overline{\overline{x}}=x$ for all $x \in B$.  
\end{enumroman}

An involution $\overline{\phantom{x}}$ is \emph{standard} if $x\overline{x} \in R$ for all $x \in B$.  
\end{defn}

Let $\overline{\phantom{x}}:B \to B$ be a standard involution on $B$.  We define the \emph{reduced trace} $\trd:B \to R$ by $\trd(x)=x + \overline{x}$ and the \emph{reduced norm} $\nrd:B \to R$ by $\nrd(x)=\overline{x}x$ for $x \in B$.  We have $x^2-\trd(x)x+\nrd(x)=0$ for all $x \in B$.  For $x,y \in B$ we have
\begin{equation} \label{xyyx}
xy+yx=\trd(y)x+\trd(x)y+\nrd(x+y)-\nrd(x)-\nrd(y).
\end{equation}

A \emph{quadratic} (rank $2$) $R$-algebra has a unique standard involution, and consequently if $B$ has a standard involution, then this involution is unique.

\begin{defn} \label{exceptdefn}
An $R$-algebra $B$ is an \emph{exceptional ring} if there is a left ideal $M \subset B$ such that $B=R \oplus M$ and the map $M \to \End_R(M)$ given by left multiplication factors through a linear map $t:M \to R$.
\end{defn}

The map $\overline{\phantom{x}}:M \to M$ by $x \mapsto t(x)-x$ defines a standard involution on $B$ and by uniqueness $t=\trd$ on $M$.  

\begin{lem} \label{exceptlocalold}
If $B$ is an exceptional ring, then the splitting $B=R \oplus M$ which realizes $B$ as an exceptional ring is unique.  Moreover, $B$ is exceptional if and only if $B_\frakp$ is exceptional for all primes $\frakp$ of $R$.
\end{lem}

We now turn to a discussion of Clifford algebras.  Let $M,N$ be projective (finitely generated) $R$-modules.  A \emph{quadratic map} is a map $q:M \to N$ satisfying:
\begin{enumroman}
\item $q(rx)=r^2q(x)$ for all $r \in R$ and $x \in M$; and
\item The map $T:M \times M \to N$ defined by
\[ T(x,y) = q(x+y)-q(x)-q(y) \]
is $R$-bilinear.
\end{enumroman}
Condition (ii) is equivalent to
\begin{equation} \label{qxyz}
q(x+y+z)=q(x+y)+q(x+z)+q(y+z)-q(x)-q(y)-q(z)
\end{equation}
for all $x,y,z \in M$.  A quadratic map $q: M \to N$ is equivalently a section of $\Sym^2 M\spcheck \otimes N$; see Wood \cite[Chapter 2]{MWood} for further discussion.  Note that $T(x,x)=2q(x)$ so if $2$ is a nonzerodivisor in $R$ we may recover $q$ from $T$.

A \emph{quadratic module} over $R$ is a triple $(M,I,q)$ where $M,I$ are projective $R$-modules with $\rk(I)=1$ and $q:M\to I$ is a quadratic map.  An \emph{isometry} between quadratic modules $(M,I,q)$ and $(M',I',q')$ is a pair of $R$-module isomorphisms $f:M \xrightarrow{\sim} M'$ and $g:I \xrightarrow{\sim} I'$ such that $q'(f(x))=g(q(x))$ for all $x \in M$, i.e., such that the diagram
\[
\xymatrix{
M  \ar[r] \ar[d]^{f}_{\wr} & I \ar[d]_{\wr}^{g} \\
M' \ar[r] & I'
} \]
commutes.

We now construct the Clifford algebra associated to a quadratic module, following Bichsel and Knus \cite[\S 3]{BichselKnus}.  (For a detailed treatment of the Clifford algebra when $N=R$ see also Knus \cite[Chapter IV]{Knus}.)  Let $(M,I,q)$ be a quadratic module over $R$.  Let $I\spcheck=\Hom(I,R)$ be the dual of the invertible $R$-module $I$, and abbreviate $I^0=R$, $I^d = \underbrace{I \otimes \dots \otimes I}_{d}$ for $d \in \Z_{>0}$, and $I^{-d} = (I\spcheck)^{\otimes d}$ for $d \in Z_{<0}$.  The $R$-module
\[ L[I] = \bigoplus_{d \in \Z} I^{\otimes d} \]
with multiplication given by the tensor product and the canonical isomorphism 
\begin{align*}
I \otimes I\spcheck &\xrightarrow{\sim} R \\
x \otimes f &\mapsto f(x)
\end{align*}
has the structure of a (commutative) $R$-algebra.  We call $L[I]$ the \emph{Rees algebra} of $I$.

Let
\[ \Ten(M)=\bigoplus_{d=0}^{\infty} M^{\otimes d} \] 
be the tensor algebra of $M$.  Let $J(q)$ denote the two-sided ideal of $\Ten(M) \otimes L[I]$ generated by elements 
\begin{equation} \label{xx1qx}
x \otimes x \otimes 1 - 1 \otimes q(x)
\end{equation}
for $x \in M$.

The algebra $\Ten(M)$ has a natural $\Z_{\geq 0}$-grading (with $\delta(x)=1$ for $x \in M$); doubling the natural grading on $L[I]$, so that $\delta(a)=2$ for $a \in I$, we find that $J(q)$ has a $\Z$-grading.  Thus, the algebra 
\[ C(M,I,q) = (\Ten(M) \otimes L[I])/J(q) \]
is also $\Z$-graded; we call $C(M,I,q)$ the \emph{Clifford algebra} of $(M,I,q)$.  

Let $C^0(M,I,q)$ denote the $R$-subalgebra of $C(M,I,q)$ consisting of elements in degree $0$; we call $C^0(M,I,q)$ the \emph{even Clifford algebra} of $(M,I,q)$.  The $R$-algebra $C^0(M,I,q)$ has rank $2^{n-1}$ where $n=\rk M$ \cite[Proposition 3.5]{BichselKnus}.  Indeed, if $I$ is free over $R$, generated by $a$, then we have a natural isomorphism 
\[ C^0(M,I,q) \cong C^0(M, R, f \circ q) \] 
where $f=a\spcheck \in I\spcheck=\Hom(I,R)$ is the dual element to $a$.  If further $M$ is free over $R$ with basis $e_1,\dots,e_n$, then $C^0(M,I,q)$ is a free $R$-algebra generated by $e_{i_1} \otimes e_{i_2} \otimes \dots \otimes e_{i_{2m}} \otimes f^m$ with $1 \leq i_1 < i_2 < \dots < i_{2m} \leq n$.  

We write $e_1e_2 \cdots e_d$ for the image of $e_1 \otimes e_2 \otimes \dots \otimes e_d \otimes 1$ in $C(M,I,q)$, for $e_i \in M$.  A standard computation gives
\begin{equation} \label{cliffordmultiply}
xy + yx = T(x,y) \in C(M,I,q)
\end{equation} 
for all $x,y \in M$, where $T$ is the bilinear form associated to $q$.

The \emph{reversal map} defined by 
\[ x=e_1e_2\cdots e_d \mapsto \overline{x}=e_d \cdots e_2e_1 \] 
is an involution on $C(M,I,q)$ which restricts to an involution on $C^0(M,I,q)$.

\begin{rmk} \label{notenough}
The reversal map $\overline{\phantom{x}}:C(M,I,q) \to C(M,I,q)$ has the property that $x\overline{x} \in R$ for all pure tensors $x=e_1e_2 \cdots e_d$, so in particular for all $x \in M$; in fact, however, in general it defines a standard involution on $C(M,I,q)$ if and only if $\rk(M) \leq 2$.  Indeed, for any $x,y,z \in M$, applying (\ref{cliffordmultiply}) we have
\begin{align*} 
(x+yz)(\overline{x+yz}) &=(x+yz)(x+zy) = q(x) + yzx + xzy + q(y)q(z) \\
&= q(x)+q(y)q(z) - T(x,y)z + T(x,z)y + T(y,z)x.
\end{align*}
Suppose that $\overline{\phantom{x}}: C(M,I,q) \to C(M,I,q)$ is a standard involution and $\rk(M) \geq 3$.  If $x,y,z$ are $R$-linearly independent, then we must have $T(x,y)=T(x,z)=T(y,z)=0$.  Moreover, the fact that $(x+1)(x+1)=q(x)+1+2x$ for all $x \in M$ implies that $2=0 \in R$.  In other words, the reversal map is a standard involution on the full Clifford algebra $C(M,I,q)$ if and only if $2=0$ in $R$ and $C(M,I,q)$ is commutative.  
\end{rmk}

Now let $(M,I,q)$ be a ternary quadratic module, so that $M$ has rank $3$.  Then by the above, the even Clifford algebra $C^0(M,I,q)$ is an $R$-algebra of rank $4$.  Explicitly, we have
\begin{equation} \label{whoocliffnelly}
C^0(M,I,q) \cong \frac{R \oplus \bigl( M \otimes M \otimes I\spcheck \bigr)}{J^0(q)}
\end{equation}
where $J^0(q)$ is the $R$-module generated by elements of the form
\[ x \otimes x \otimes f - 1 \otimes f(q(x)) \]
for $x \in M$ and $f \in I\spcheck$.  The reversal map resticts to a standard involution on $C^0(M,I,q)$ (proved using similar calculations as in Remark \ref{notenough}, or see Example \ref{Qfree}).  The association $(M,I,q) \to C^0(M,I,q)$ is functorial with respect to isometries of quadratic modules and so we refer to it as the \emph{Clifford functor}.

\begin{exm} \label{Qfree}
Given the free module $M=R^3=Re_1 \oplus Re_2 \oplus Re_3$ equipped with the quadratic form $q:M \to R$ by
\begin{equation} \label{abcuvw}
q(xe_1+ye_2+ze_3)=q(x,y,z)=ax^2+by^2+cz^2+uyz+vxz+wxy,
\end{equation}
with $a,b,c,u,v,w \in R$, we compute directly that the Clifford algebra of $M$ is given by 
\[ C^0(M,q)=R \oplus Re_2e_3 \oplus Re_3e_1 \oplus Re_1e_2 \] 
and the map
\begin{equation} \label{BtoC}
\begin{aligned}
B &\xrightarrow{\sim} C^0(M,I,q) \\
i,j,k &\mapsto e_2e_3,e_3e_1,e_1e_2
\end{aligned}
\end{equation}
gives an isomorphism to the algebra $B$ where the following multiplication laws hold.  
\begin{align*}
i^2 &= ui-bc & jk &= a\overline{i}=a(u-i) & kj &= \overline{\overline{j}\,\overline{k}} = -vw + ai + wj + vk \\
j^2 &= vj-ac & ki &= b\overline{j}=b(v-j) & ik &= \overline{\overline{k}\,\overline{i}} = -uw + wi + bj + uk \tag{$Q$} \\
k^2 &= wk-ab & ij &= c\overline{k}=c(w-k) & ji &= \overline{\overline{i}\,\overline{j}} = -uv + vi + uj + ck 
\end{align*}
This construction has been attributed to Eichler and appears in Brzezinski \cite{Brzezinski} in the case $R=\Z$.  In this association, the reversal map corresponds to the standard involution $\overline{\phantom{x}}$ on $B$.  

Such a free quaternion ring is commutative \cite[\S 3]{Voightlowrank} if and only if either 
\[ B \cong R[i,j,k]/(i,j,k)^2 \] 
or 
\[ B \cong R[i,j,k]/(i^2+bc,j^2+ac, k^2+ab, jk+ai, ki+bj, ij+ck) \] 
with $a,b,c \in R$ satisfying $2a=2b=2c=0$.
\end{exm}

\begin{defn}
A \emph{quaternion ring} over $R$ is an $R$-algebra $B$ such that $B \cong C^0(M,I,q)$ with $(M,I,q)$ a ternary quadratic module over $R$.
\end{defn}

Note that if $R \to S$ is any ring homomorphism and $B$ is a quaternion ring, then $B_S=B \otimes_R S$ is also a quaternion ring by functoriality of the Clifford algebra construction.  

\section{The canonical exterior form}

We have seen in the previous section that the even Clifford functor associates to a ternary quadratic module an $R$-algebra of rank $4$ with a standard involution.  In this section, we show how to do the converse; we will then show in Section 3 that this indeed furnishes an inverse to the Clifford functor on its image.  Throughout this section, let $B$ be an $R$-algebra of rank $4$ with a standard involution.

Following Bhargava \cite{BhargavaQuartic} (who considered the case of commutative rings of rank $4$) and a footnote of Gross and Lucianovic \cite[Footnote 2]{GrossLuc}, we define the following quadratic map.

First, we note that $R$ is a direct summand of $B$, so $B/R$ is a projective $R$-module \cite[Lemma 1.3]{Voightlowrank}.  It follows that $\tbigwedge^4 B \cong \tbigwedge^3 (B/R)$ canonically.  

\begin{lem} \label{xwedgey}
There exists a unique quadratic map 
\[ \phi_B:\tbigwedge^2 (B/R) \to \tbigwedge^4 B \]
with the property that
\[ \phi_B(x \wedge y) = 1 \wedge x \wedge y \wedge xy \]
for all $x,y \in B$.
\end{lem}

\begin{proof}
We first define the map on sets $\varphi:B \times B \to \bigwedge^4 B$ by $(x,y) \mapsto 1 \wedge x \wedge y \wedge xy$, where $B \times B$ denotes the Cartesian product.  This map descends to a map from $B/R \times B/R$.  We have $\varphi(ax,y)=\varphi(x,ay)$ for all $x,y \in B$ and $a \in R$.  Furthermore, we have $\varphi(x,x)=0$ for all $x \in B$ and by (\ref{xyyx}) we have
\begin{equation} \label{fxyyx}
\varphi(y,x)=1\wedge y\wedge x \wedge yx = -1\wedge x\wedge y \wedge (-xy) = \varphi(x,y)=\varphi(x,-y)
\end{equation}
for all $x,y \in B$.  Finally, the map $\varphi$ when restricted to each variable $x,y$ separately yields a quadratic map $B/R \to \bigwedge^4 B$.

We now prove the existence and uniqueness of the map $\phi=\phi_B$ when $B$ is free.  Let $i,j,k \in B$ form a basis for $B/R$.  Then $i \wedge j,j \wedge k, k \wedge i$ is a basis for $\bigwedge^2 (B/R)$.  It follows from (\ref{qxyz}) that to define a quadratic map $q:M \to N$ on a free module $M$ is equivalent to choosing elements $q(x),q(x+y) \in N$ for $x,y$ in any basis for $M$.  We thereby define
\begin{equation} \label{needrk3}
\begin{aligned}
\phi: \tbigwedge^2 (B/R) &\to \tbigwedge^4 B  \\
\phi(i \wedge j) &= \varphi(i,j) \\
\phi(i \wedge j + j \wedge k) &= \varphi(i-k,j) = \varphi(j,k-i)
\end{aligned}
\end{equation}
together with the cyclic permutations of (\ref{needrk3}).  By construction, the map $\phi$ is quadratic.

Now we need to show that in fact $\phi(x \wedge y)=\varphi(x,y)$ for all $x,y \in B$.  By definition and (\ref{fxyyx}), we have that this is true if $x,y \in \{i,j,k\}$.  For any $y \in \{i,j,k\}$, consider the maps
\begin{align*} 
\varphi_y,\phi_y : B/R &\to \tbigwedge^4 B \\
x &\mapsto \varphi(x \wedge y), \phi(x \wedge y)
\end{align*}
restricted to the first variable.  Note that each of these maps are quadratic and they agree on the values $i,j,k,i-k,j-i,k-j$, so they are equal.  The same argument on the other variable, where now we may restrict $\varphi,\phi$ with any $x \in B$, gives the result.  

To conclude, for any $R$-algebra $B$ there exists a finite cover of standard open sets $\{\Spec R_f\}_f$ of $\Spec R$ with $f \in R$ such that each localization $B_f$ is free.  By the above constructions, we have a map on each $B_f$ and by uniqueness these maps agree on overlaps, so by gluing we obtain a unique map $\phi$.
\end{proof}

\begin{rmk}
We used in (\ref{needrk3}) that $B$ has rank $4$; indeed, if $\rk(B)> 4$, there will be many ways to define such a map $\phi$.
\end{rmk}

We call the map $\phi_B:\tbigwedge^2 (B/R) \to \tbigwedge^4 B$ in Lemma \ref{xwedgey} the \emph{canonical exterior form} of $B$.

\begin{exm} \label{canonextfree}
Let $B$ be a free quaternion ring with basis $i,j,k$ and multiplication laws as in ($Q$) in Example \ref{Qfree}.  We compute the canonical exterior form 
\[ \phi=\phi_B:\tbigwedge^2(B/R) \to \tbigwedge^4 B \]
directly.  We have isomorphisms $\tbigwedge^4 B \to R$ by $1 \wedge i \wedge j \wedge k \mapsto -1$ and
\[ \tbigwedge^2(B/R) \xrightarrow{\sim} R(j \wedge k) \oplus R(k \wedge i) \oplus R(i \wedge j) = R e_1 \oplus Re_2 \oplus Re_3. \]
With these identifications, the canonical exterior form $\phi : R^3 \to R$ has
\[ \phi(e_1)=\phi(j \wedge k)=1 \wedge j \wedge k \wedge jk = 1 \wedge j \wedge k \wedge (-ai) \mapsto a \]
and 
\begin{align*} 
\phi(e_1+e_2)-\phi(e_1)-\phi(e_2) &= \phi(k\wedge (i-j)) - \phi(j \wedge k) - \phi(k \wedge i) \\
&= -1 \wedge k \wedge j \wedge ki - 1 \wedge k \wedge i \wedge kj \\
&= -w(1 \wedge k \wedge i \wedge j) \mapsto w.
\end{align*}
In this way, we see directly that $\phi(x(j \wedge k) + y(k \wedge i) + z(i \wedge j))=q(xe_1+ye_2+ze_3)$ is identified with the form in (\ref{abcuvw}).
\end{exm}

\begin{exm}
Suppose that $R$ is a Dedekind domain with field of fractions $F$.  Then we can write 
\begin{equation} \label{goodpseudo}
B = R \oplus \fraka i \oplus \frakb j \oplus \frakc k 
\end{equation}
with $\fraka,\frakb,\frakc \subset F$ fractional $R$-ideals and $i,j,k \in B$.  By the same reasoning as in the free case, we may assume that $1,i,j,k$ satisfy the multiplication rules ($Q$), and then we say that the decomposition (\ref{goodpseudo}) is a \emph{good pseudobasis}, and the canonical exterior form of $B$ is, analogously as in Example \ref{canonextfree}, given by
\[ \phi_B : \frakb\frakc e_1 \oplus \fraka\frakc e_2 \oplus \fraka\frakb e_3 \to \fraka\frakb\frakc \]
under the identification $\tbigwedge^4 B \xrightarrow{\sim} \fraka\frakb\frakc$ induced by $1 \wedge i \wedge j \wedge k \mapsto -1$; here, $\phi_B(xe_1+ye_2+ze_3)$ is given as in (\ref{abcuvw}) but now with $x,y,z$ in their respective coefficient ideals.
\end{exm}

\begin{exm} \label{canonexceptfree}
Let $B$ be a free exceptional ring $B=R \oplus M$ where $M$ has basis $i,j,k$.  Then by definition we have the following multiplication laws in $B$:
\begin{align*}
i^2 &= ui & jk &= vk & kj &= wj \\
j^2 &= vj & ki &= wi & ik &= uk \tag{$E$} \\
k^2 &= wk & ij &= uj & ji &= vi.
\end{align*}
We then compute as in the previous example that the canonical form $\phi_B$ is identically zero.

Comparing ($Q$) to ($E$), we see that a ring is both a quaternion ring and an exceptional ring; if $2$ is a nonzerodivisor in $R$, then this holds if and only if $B$ is commutative.
\end{exm}

\begin{rmk}
From the above example, we see directly that there exists a bijection between the set of orbits of $GL(R^3)$ on $R^3$ and the set of isomorphism classes of free exceptional rings of rank $4$, where we associate to the triple $(u,v,w) \in R^3$ the algebra with multiplication laws as in ($E$).

Lucianovic \cite[Proposition 1.8.1]{Luc} instead associates to $(u,v,w) \in R^3$ the skew-symmetric matrix $M=\begin{pmatrix} 0 & w & -v \\ -w & 0 & u \\ v & -u & 0 \end{pmatrix}$, and $g \in GL_3(R)$ acts on $M$ by $M \mapsto (\det g) ({}^t g)^{-1} M g^{-1}$.  This more complicated (but essentially equivalent) association gives a bijection to the set of orbits of $GL(N)$ on $\bigwedge^2 N\spcheck \otimes \bigwedge^3 N$.  
\end{rmk}

Using the canonical exterior form, we can distinguish quaternion rings and exceptional rings in the class of algebras of rank $4$ with a standard involution in the final section.

\section{An equivalence of categories}

In this section, we prove Theorem B, generalizing the following result of Gross and Lucianovic.  

\begin{prop}[Gross-Lucianovic] \label{GrossLucN3}
Let $M$ be a free module of rank $3$.  Then there is a bijection between the set of orbits $GL(M)$ on $\Sym^2(M\spcheck) \otimes \bigwedge^3 M$ and the set of isomorphism classes of free quaternion rings over $R$.
\end{prop}

In this bijection, a quadratic form $q(x,y,z)$ as in (\ref{abcuvw}) is associated to the even Clifford algebra $C^0(R^3,q)$ as in Example \ref{Qfree}.  

The bijection of Proposition \ref{GrossLucN3} is \emph{discriminant-preserving}, as follows.  We define the \emph{(half-)discriminant} of a quadratic form $q(x,y,z)$ as in (\ref{abcuvw}) by
\begin{equation} \label{Dq}
D(q)=4abc+uvw-au^2-bv^2-cw^2.
\end{equation}
 On the other hand, we define the (reduced) \emph{discriminant} $D(B)$ of an algebra $B$ of rank $4$ with standard involution to be the ideal of $R$ generated by all values 
\begin{equation} \label{Dxyz}
\{x,y,z\} = \trd([x,y]\overline{z})=\trd((xy-yx)\overline{z})
\end{equation}
with $x,y,z \in B$.  If $1,i,j,k$ is a good basis for $B$, a direct calculation verifies that already
\[ \{i,j,k\}=-D(q) \]
so the map preserves discriminants.  In particular, every exceptional ring $B$ with good basis $i,j,k$ has $\{i,j,k\}=0$ so that $D(B)=0$, so if one restricts to $R$-algebras $B$ with $D(B) \neq 0$ one will never see an exceptional ring, and it is perhaps for this reason that they fail to appear in more classical treatments.

We warn the reader that although the even Clifford algebra construction which gives rise to the bijection in Proposition \ref{GrossLucN3} is functorial with respect to isometries and isomorphisms, respectively, it is not always functorial with respect to other morphisms.

\begin{exm}
Consider the sum of squares form $q(x,y,z)=x^2+y^2+z^2$ over $R=\Z$.  The associated quaternion ring $B$ is generated over $\Z$ by the elements $i,j,k$ subject to $i^2=j^2=k^2=-1$ and $ijk=-1$ and has discriminant $4$.  The ring $B$ is an order inside the quaternion algebra of discriminant $2$ over $\Q$ which gives rise to the Hamiltonian ring over $\R$, and $B$ is contained in the maximal order $B_{\text{max}}$ (of discriminant $2$) obtained by adjoining the element $(1+i+j+k)/2$ to $B$.  Indeed, the ring $B_{\text{max}}$ is obtained from the Clifford algebra associated to the form $q_{\text{max}}(x,y,z)=x^2+y^2+z^2+yz+xz+yz$ of discriminant $2$.  However, the lattice associated to the form $q$ is maximal in $\Q^3$, so there is no inclusion of quadratic modules which gives rise to the inclusion $B \hookrightarrow B_{\text{max}}$ of these two quaternion orders.  Of course, if we consider the correspondence over $\Q$ (or even $\Z[1/2]$ then the inclusion becomes an isomorphism, which shows that the lattice associated to each of the quadratic forms $q$ and $q_{\text{max}}$ are contained in the same quadratic space over $\Q$, and the change of basis from one to the other is defined over $GL_3(\Q)$ (but not $GL_3(\Z)$).
\end{exm}

\begin{rmk}
There is an alternative association between forms and algebras which we call the \emph{trace zero method} and describe for the sake of comparison (see also Lucianovic \cite[Remark, pp.~28--29]{Luc}).  Let $B$ be a free $R$-algebra of rank $4$ with a standard involution and let $B^0=\{x \in B : \trd(x)=0\}$ be the elements of reduced trace zero in $B$.  Then $(B^0,\nrd|_{B^0})$ is a ternary quadratic module.  

Starting with a quadratic form $(R^3,q)$, considering the free quaternion algebra $B=C^0(R^3,q)$ with good basis as in (\ref{BtoC}), the trace zero module $(B^0,\nrd)$ has basis $jk-kj,ki-ik,ij-ji$ and we compute that
\[ \nrd(x(jk-kj)+y(ki-ik)+z(ij-ji))=D(q)q(x,y,z). \]
In particular, if $D(q)=D(B) \in R^*$, in which case $q$ is said to be \emph{semiregular}, we can instead associate to $B$ the quadratic module $(B^0, -D(B)^{-1} \nrd)$ to give an honest bijection.  One can use this together with localization to prove a result for an arbitrary quadratic module $(M,q)$, as exihibited by Knus \cite[\S V.3]{KnusQHFR}.  This point of view works very well, for example, in the classical case of quaternion algebras over a field.  When the discriminant of $(M,q)$ is principal and $R$ is a domain, one can similarly adjust the maps to obtain a bijection \cite{Brzezinski}.  However, in general it is not clear how to generalize this method to quadratic forms which are not semiregular.
\end{rmk}

It is perhaps tempting to think that we will simply find a functorial bijection between isometry classes of ternary quadratic modules over $R$ and isomorphism classes of quaternion rings over $R$; however, we notice one obstruction which does not appear in the free case.  

Let $(M,I,q)$ be a ternary quadratic module.  Recall the definition of the even Clifford algebra $C^0(M,I,q)$ from Section 1.  By definition, as an $R$-module, we have 
\begin{equation} \label{wedgezero}
C^0(M,I,q)/R \cong \tbigwedge^2 M \otimes I\spcheck.
\end{equation}
To analyze this isomorphism, we note the following lemma.

\begin{lem} \label{wedgewedge}
Let $M$ be a projective $R$-module of rank $3$.  Then there are canonical isomorphisms
\begin{equation} \label{wedgewedge1}
\tbigwedge^3 \bigl(\tbigwedge^2 M\bigr) \xrightarrow{\sim} \bigl(\tbigwedge^3 M\bigr)^{\otimes 2}
\end{equation}
and
\begin{equation} \label{wedgewedge2}
\tbigwedge^2(\tbigwedge^2 M) \xrightarrow{\sim} M \otimes \tbigwedge^3 M.
\end{equation}
\end{lem}

\begin{proof}
We exhibit first the isomorphism (\ref{wedgewedge1}).  We define the map
\begin{align*} 
s:M^{\otimes 6} &\to \bigl(\tbigwedge^3 M\bigr)^{\otimes 2} \\
x \otimes x' \otimes y \otimes y' \otimes z \otimes z' &\mapsto (x \wedge x' \wedge y') \otimes (y \wedge z \wedge z') \\
&\hspace{12ex} - (x \wedge x' \wedge y) \otimes (y' \wedge z \wedge z')
\end{align*}
with $x,x',y,y',z,z'\in M$.

It is easy to see that $s$ descends to $(\tbigwedge^2 M)^{\otimes 3}$; we show that $s$ in fact descends to $\tbigwedge^3 (\tbigwedge^2 M)$.  We observe that 
\[ s(x \wedge x' \otimes y \wedge y' \otimes z \wedge z')=0 \]
 whenever $x=y$ and $x'=y'$ (with similar statements for $x,z$ and $y,z$).  To finish, we show that
\begin{equation} \label{switchxyz}
s((x \wedge x') \otimes (y \wedge y') \otimes (z \wedge z')) = -s((y \wedge y') \otimes (x \wedge x') \otimes (z \wedge z')).
\end{equation}
To prove (\ref{switchxyz}) we may do so locally and hence assume that $M$ is free with basis $e_1,e_2,e_3$; by linearity, it is enough to note that
\begin{align*} 
s((e_1 \wedge e_2) \otimes (e_2 \wedge e_3) \otimes (e_3 \wedge e_1)) &= 
(e_1 \wedge e_2 \wedge e_3) \otimes (e_2 \wedge e_3 \wedge e_1) \\
&= (e_2 \wedge e_3 \wedge e_1) \otimes (e_2 \wedge e_3 \wedge e_1) \\
&= -s((e_2 \wedge e_3) \otimes (e_1 \wedge e_2) \otimes (e_3 \wedge e_1)).
\end{align*}
It follows then also that $s$ is an isomorphism, since it maps the generator
\[ (e_1 \wedge e_2) \wedge (e_2 \wedge e_3) \wedge (e_3 \wedge e_1) \in \tbigwedge^3(\tbigwedge^2 M) \] 
to the generator $(e_1 \wedge e_2 \wedge e_3) \otimes (e_2 \wedge e_3 \wedge e_1) \in (\tbigwedge^3 M)^{\otimes 2}$.

The second isomorphism (\ref{wedgewedge2}) arises from the map
\begin{equation} \label{ww2map}
\begin{aligned}
M^{\otimes 4} &\to M \otimes \tbigwedge^3 M \\
x \otimes x' \otimes y \otimes y' &\mapsto x' \otimes (x \wedge y \wedge y') - x \otimes (x' \wedge y \wedge y')
\end{aligned}
\end{equation}
and can be proved in a similar way.
\end{proof}

By (\ref{wedgewedge1}) and (\ref{wedgezero}), we find that
\begin{equation} \label{oopsnotPic}
\tbigwedge^4 C^0(M,I,q) \cong \tbigwedge^3 (C^0(M,I,q)/R) \cong \tbigwedge^3 \bigl(\tbigwedge^2 M \otimes I\spcheck \bigr) \cong \bigl(\tbigwedge^3 M\bigr)^{\otimes 2} \otimes (I\spcheck)^{\otimes 3}.
\end{equation}
(Compare this with Kable et al.\  \cite{Kable}, who considers the Steinitz class of a central simple algebra over a number field, and Peters \cite{Peters} who works over a Dedekind domain.)

Cognizant of (\ref{oopsnotPic}), we make the following definition.  Let $N$ be an invertible $R$-module.  A \emph{parity factorization} of $N$ is an $R$-module isomorphism 
\[ p:P^{\otimes 2} \otimes Q \xrightarrow{\sim} N \]
where $P,Q$ are invertible $R$-modules.  Note that $N$ always has the \emph{trivial} parity factorization $R^{\otimes 2} \otimes N \xrightarrow{\sim} N$.  An isomorphism between two parity factorizations $p:P^{\otimes 2} \otimes Q \xrightarrow{\sim} N$ and $p':P'^{\otimes 2} \otimes Q' \xrightarrow{\sim} N'$ is given by isomorphism $P \xrightarrow{\sim} P'$, $Q \xrightarrow{\sim} Q'$, $N \xrightarrow{\sim} N'$ which commute with $p,p'$.

We are now ready for the main result in this section.

\begin{thm} \label{mainthmcorr}
There is a bijection
\begin{equation*}
\left\{ \begin{minipage}{28ex} 
\begin{center}
Isometry classes of ternary \\
quadratic modules $(M,I,q)$ \\
over $R$
\end{center} 
\end{minipage}
\right\} \longleftrightarrow \left\{ \begin{minipage}{39ex} 
\begin{center}
Isomorphism classes of quaternion \\
rings $B$ over $R$ equipped with a parity \\
factorization $p:P^{\otimes 2} \otimes Q \xrightarrow{\sim} \tbigwedge^4 B$
\end{center} 
\end{minipage}
\right\}
\end{equation*}
which is functorial with respect to the base ring $R$.  In this bijection, the isometry class of a quadratic module $(M,I,q)$ maps to the isomorphism class of the quaternion ring $C^0(M,I,q)$ equipped with the parity factorization
\begin{equation} \label{yeahparityfact}
(\tbigwedge^3 M \otimes (I\spcheck)^{\otimes 2})^{\otimes 2} \otimes I \xrightarrow{\sim} \tbigwedge^4 C^0(M,I,q).
\end{equation}
\end{thm}

In other words, there is a functor (the even Clifford algebra functor) which associates to each ternary quadratic module over $R$ a quaternion ring over $R$ equipped with a parity factorization, and given a ring homomorphism $R \to S$ there is a natural transformation between the two functors restricted to objects over $R$ and $S$, respectively.  This functor gives a bijection over $R$ when one looks at the corresponding set of isomorphism classes on each side.

\begin{proof}
Given a ternary quadratic module $(M,I,q)$, we associate to it the even Clifford algebra $B=C^0(M,I,q)$ with the parity factorization (\ref{yeahparityfact}) as in (\ref{oopsnotPic}).  The algebra $B$ is a quaternion ring by definition.  

In the other direction, we use the canonical exterior form $\phi_B : \tbigwedge^2 (B/R) \to \tbigwedge^4 B$ as defined in (\ref{xwedgey}).  Let $B$ be a quaternion ring with parity factorization $p:P^{\otimes 2} \otimes Q \xrightarrow{\sim} \tbigwedge^4 B$.  Then by dualizing, the map $p$ gives an isomorphism 
\[ p^{*}:(P\spcheck)^{\otimes 2} \xrightarrow{\sim} (\tbigwedge^4 B)\spcheck \otimes Q. \]
Note that $p^{*}$ defines a quadratic map $P\spcheck \to (\tbigwedge^4 B)\spcheck \otimes Q$ by $x \mapsto p^*(x \otimes x)$.  We associate then to the pair $(B,p)$ the ternary quadratic module associated to the quadratic map
\begin{equation} \label{phiBpstar}
\phi_B \otimes p^{*} : \tbigwedge^2(B/R) \otimes P\spcheck \to \bigwedge^4 B \otimes \bigl((\tbigwedge^4 B)\spcheck \otimes Q\bigr) \xrightarrow{\sim} Q.
\end{equation}

We need to show that these associations are indeed inverse to each other.  First, given the algebra $C^0(M,I,q)$ with parity factorization $p$ as in (\ref{yeahparityfact}), we have by the above association the ternary quadratic module
\begin{equation} \label{phipstar}
\phi \otimes p^{*} : \tbigwedge^2(C^0(M,I,q)/R) \otimes (\tbigwedge^3 M)\spcheck \otimes I^{\otimes 2} \to I.
\end{equation}
From (\ref{wedgezero}) and (\ref{wedgewedge2}) we obtain
\[ \tbigwedge^2(C^0(M,I,q)/R) \cong \tbigwedge^2\bigl(\tbigwedge^2 M \otimes I\spcheck\bigr) \cong \tbigwedge^2 (\tbigwedge^2 M) \otimes (I\spcheck)^{\otimes 2} \cong M \otimes \tbigwedge^3 M \otimes (I\spcheck)^{\otimes 2}
\]
hence the ternary quadratic module $\phi \otimes p^*$ (\ref{phipstar}) has domain canonically isomorphic to
\[ \bigl(M \otimes \tbigwedge^3 M \otimes (I\spcheck)^{\otimes 2}\bigr) \otimes (\tbigwedge^3 M)\spcheck \otimes I^{\otimes 2} \cong M \]
and so yields a quadratic map $\phi \otimes p^*:M \to I$.  

To show that $q$ is isometric to $\phi \otimes p^*$ via the above canonical isomorphisms, we may do so locally and therefore assume that $M,I$ are free so that $q:R^3 \to R$ is given as in (\ref{abcuvw}).  Then the Clifford algebra $B=C^0(R^3,q)$ is a quaternion ring defined by the multiplication rules ($Q$).  By Example \ref{canonextfree}, we indeed have an isometry between $\phi_B$ and $q$, as desired.

The other direction is proved similarly.  Beginning with an $R$-algebra $B$ with a parity factorization $p:P^{\otimes 2} \otimes Q \xrightarrow{\sim} \bigwedge^4 B$, we associate the quadratic map $\phi_B \otimes p^*$ as in (\ref{phiBpstar}); to this, we associate the Clifford algebra $C^0(\tbigwedge^2(B/R) \otimes P\spcheck, Q, \phi_B \otimes p^*)$, which we abbreviate simply $C^0(B)$, with parity factorization
\begin{equation} \label{wootpf}
\tbigwedge^4 C^0(B) \xrightarrow{\sim}
\bigl( \tbigwedge^3(\tbigwedge^2(B/R) \otimes P\spcheck) \otimes (Q\spcheck)^{\otimes 2})^{\otimes 2} \otimes Q.
\end{equation}
From (\ref{wedgewedge1}) we obtain the canonical isomorphism
\begin{align*} 
\tbigwedge^3(\tbigwedge^2(B/R) \otimes P\spcheck) &\cong \tbigwedge^3(\tbigwedge^2(B/R)) \otimes (P\spcheck)^{\otimes 3} \\
&\cong \bigl(\tbigwedge^3(B/R)\bigr)^{\otimes 2} \otimes (P\spcheck)^{\otimes 3} \cong (\tbigwedge^4 B)^{\otimes 2} \otimes (P\spcheck)^{\otimes 3}.
\end{align*}
But now applying the original parity factorization $p:P^{\otimes 2} \otimes Q \xrightarrow{\sim} \tbigwedge^4 B$, we obtain
\[ (\tbigwedge^4 B)^{\otimes 2} \otimes (P\spcheck)^{\otimes 3} \cong (P^{\otimes 2} \otimes Q)^{\otimes 2} \otimes (P\spcheck)^{\otimes 3} \cong P \]
so putting these together, the parity factorization (\ref{wootpf}) becomes simply 
\[ \tbigwedge^4 C^0(B) \cong P^{\otimes 2} \otimes Q. \]

Similarly, putting together (\ref{wedgezero}), (\ref{wedgewedge2}), and the dual isomorphism $p\spcheck$ to $p$, we have
\begin{equation} \label{CRBR}
\begin{aligned} 
C^0(B)/R &= C^0(\tbigwedge^2(B/R) \otimes P\spcheck, Q, \phi_B \otimes p^*)/R \\
&\cong \tbigwedge^2\bigl(\tbigwedge^2(B/R) \otimes P\spcheck \bigr) \otimes Q\spcheck \\
&\cong \tbigwedge^2\bigl(\tbigwedge^2(B/R)\bigr) \otimes (P\spcheck)^{\otimes 2} \otimes Q\spcheck \\
&\cong B/R \otimes \tbigwedge^3(B/R) \otimes (\tbigwedge^4 B)\spcheck \cong B/R.
\end{aligned}
\end{equation}
We now show that there is a unique isomorphism $C^0(B) \xrightarrow{\sim} B$ of $R$-algebras which lifts the map in (\ref{CRBR}).  It suffices to show this locally; the map is well-defined up to addition of scalars, so we may assume that $B$ is free with good basis $1,i,j,k$ (and that $P,Q \cong R$).  But then with this basis it follows that the map (\ref{BtoC}) is the already the unique map which identifies $C^0(B) \cong B$, and the result follows.

In this way, we have exhibited an equivalence of categories between the category of ternary quadratic modules (with morphisms isometries) and the category of quaternion rings $B$ over $R$ equipped with a parity factorization $p$ (with morphisms isomorphisms).  It follows that the set of equivalence classes under isometry and isomorphisms are in functorial bijection.
\end{proof}

\begin{rmk} \label{reducetogrossluc}
We recover the bijection of Gross-Lucianovic (Proposition \ref{GrossLucN3}) from Theorem \ref{mainthmcorr} as follows.  Let $R$ be a PID or a local ring and let $M$ be an $R$-module of rank $3$.  An element of $\Sym^2(M\spcheck) \otimes \bigwedge^3 M$ corresponds to a quadratic map $q:M \to \bigwedge^3 M$, and an isometry between two such forms is specified by an element $g \in GL(M)$ acting as
\[
\xymatrix{
M  \ar[r]^{q} \ar[d]^{g^{-1}}_{\wr} & \bigwedge^3 M \ar[d]_{\wr}^{\wedge^3 g} \\
M \ar[r]^{q'} & \bigwedge^3 M.
} \]
The quaternion algebra $B=C^0(M,\bigwedge^3 M,q)$ is then equipped with the parity factorization 
\[ \tbigwedge^3 M\spcheck \cong (\tbigwedge^3 M \spcheck)^{\otimes 2} \otimes \tbigwedge^3 M \xrightarrow{\sim} \tbigwedge^4 C^0(M,\tbigwedge^3 M, q) = \tbigwedge^4 B. \]
Conversely, given a quaternion ring $B$ over $R$ equipped with an isomorphism $B/R \cong  \xrightarrow{\sim} (\bigwedge^3 M)\spcheck$, composing with the canonical exterior form then yields
\[ \tbigwedge^2 (B/R) \otimes \tbigwedge^3 M \to \tbigwedge^4 B \xrightarrow{\sim} \tbigwedge^3 M. \]
Therefore $GL(M)$ acts as isomorphisms on $B$ via the isomorphism $B/R \cong \tbigwedge^2 M \otimes (\tbigwedge^3 M)\spcheck$, which recovers the result of Gross and Lucianovic (see also Lucianovic \cite[\S 1.7]{Luc}) who identified this action by considering the corresponding action on \emph{good bases} (see \S 4 below).
\end{rmk}

If one wishes only to understand isomorphism classes of quaternion rings, one can consider the functor which forgets the parity factorization.  In this way, certain ternary quadratic modules will be identified.  We define a \emph{twisted discriminant module} to be a quadratic module $(P,Q,d)$ where $P,Q$ are invertible $R$-modules, or equivalently an $R$-linear map $d:P \otimes P \to Q$.  A \emph{twisted isometry} between two quadratic modules $(M,I,q)$ and $(M',I',q')$ is an isometry between $(M \otimes P, I \otimes Q, q \otimes d)$ and $(M',I',q')$ for some twisted discriminant module $(P,Q,d)$.  (For a reference, see Knus \cite[\S III.3]{Knus} or Balaji \cite{Balaji}.)

\begin{cor}
There is a functorial bijection
\begin{equation*}
\left\{ \begin{minipage}{28ex} 
\begin{center}
Twisted isometry classes of \\
ternary quadratic modules \\
$(M,I,q)$ over $R$ 
\end{center} 
\end{minipage}
\right\} \longleftrightarrow \left\{ \begin{minipage}{28ex} 
\begin{center}
Isomorphism classes of \\
quaternion rings $B$ over $R$
\end{center} 
\end{minipage}
\right\}.
\end{equation*}
\end{cor}

\begin{proof}
Given a quaternion ring $B$ over $R$, from the trivial parity factorization we obtain the ternary quadratic module $\phi_B: \tbigwedge^2(B/R) \to \tbigwedge^4 B$.  By (\ref{phipstar}), we see that the choice of an (isomorphism class of) parity factorization $p:P^{\otimes 2} \otimes Q \xrightarrow{\sim} \tbigwedge^4 B$ corresponds to twisting $\phi_B$ by $(P\spcheck, (\tbigwedge^4 B)\spcheck \otimes Q, p^*)$, and the result follows.
\end{proof}

The bijection of Theorem \ref{mainthmcorr} is also discriminant-preserving as in the free case, when the proper definitions are made.  We define the \emph{(half-)discriminant} $D(M,I,q)$ of a quadratic module $(M,I,q)$ to be ideal of $R$ generated by $D(q|_N)$ for all free (ternary) submodules $N \subset M$.  Then in this correspondence we have since $D(M,I,q)_\frakp=D(C^0(M,I,q))_\frakp$ since the bijection preserves discriminants in the local (free) case.

\section{Characterizing quaternion rings}

In this section, we consider two consequences of Theorem B.  We prove Theorem A which compares quaternion rings to exceptional rings and then turn to other certain special cases of Theorem B which appear in the literature.

Let $B$ be an $R$-algebra of rank $4$ with a standard involution.  One can compute the `universal' free such algebra $B$, following Gross and Lucianovic \cite{GrossLuc} as follows.  We say a basis $1,i,j,k$ for $B$ is \emph{good} if the coefficient of $j$ (resp.\ $k$, $i$) in $jk$ (resp.\ $ki$, $ij$) is zero.  One can add suitable elements of $R$ to any basis to turn it into a good basis.

\begin{prop} \label{grossluc}
If $B$ is free $R$-algebra of rank $4$ with a standard involution with good basis $1,i,j,k$, then
\begin{align*}
i^2 &= ui-bc & jk &= a\overline{i} + v' k \\
j^2 &= vj-ac & ki &= b\overline{j} + w' i \tag{$U$} \\
k^2 &= wk-ab & ij &= c\overline{k} + u' j
\end{align*}
with $a,b,c,u,v,w,u',v',w' \in R$ which satisfy 
\begin{equation} \label{uuvvww}
\begin{pmatrix} a \\ b \\ c \end{pmatrix} \begin{pmatrix} u' & v' & w' \end{pmatrix} = 0 \quad \text{and} \quad \begin{pmatrix} u'-u \\ v'-v \\ w'-w \end{pmatrix} \begin{pmatrix} u' & v' & w' \end{pmatrix} = 0.
\end{equation}
Conversely, any algebra defined by laws $(U)$ subject to \textup{(\ref{uuvvww})} is an algebra of rank $4$ with standard involution equipped with the good basis $1,i,j,k$.
\end{prop}

\begin{proof}
The result follows by a direct calculation, considering the consequences of the associative laws like $j(k\overline{k})=(jk)\overline{k}$ and $(ij)k=i(jk)$, the others being obtained by symmetry.  For details, see Lucianovic \cite[Proposition 1.6.2]{Luc}.  
\end{proof}

\begin{rmk} \label{uuvvwwdomain}
If $u'=v'=w'=0$ in ($U$), then $B$ is a quaternion ring over $R$ by the multiplication laws ($Q$) in Example \ref{Qfree}, and if $a=b=c=0$ and $(u,v,w)=(u',v',w')$, then $B$ is an exceptional ring over $R$ by the rules ($E$) in Example \ref{canonexceptfree}.  

Consequently, if $R$ is a domain then the equations (\ref{uuvvww}) hold if and only if one of the cases ($Q$) or ($E$) holds, so a free $R$-algebra of rank $4$ with a standard involution over a domain $R$ is either a quaternion ring or an exceptional ring \cite[Proposition 1.6.2]{Luc}.
\end{rmk}

\begin{prop} \label{closedexcept}
The set of primes $\frakp$ such that $B \otimes R/\frakp = B/\frakp B$ is a quaternion (resp.\ exceptional) ring is closed in $\Spec R$.  

Given an algebra of rank $4$ over $R$ with standard involution, there exists a decomposition $\Spec R_Q \cup \Spec R_E \hookrightarrow \Spec R$ such that the restriction $B_{R_Q}=B \otimes_R R_Q$ is a quaternion ring, $B_{R_E}$ is an exceptional ring, and $\Spec R_Q$ and $\Spec R_E$ are the largest (closed) subschemes with these properties.
\end{prop}

\begin{proof}
Covering $\Spec R$ by (basic) open sets we may assume that $B$ is free over $R$ and so has a good basis as in ($U$).  Define $\Spec R_Q=\Spec R/(u',v',w')$ and $\Spec R_E=\Spec R/(a,b,c,u'-u,v'-v,w'-w)$.  Then $B_Q = B \otimes_R R_Q$ is a quaternion ring over $R_Q$ and it is the largest such subscheme, and similarly $B_E$ is an exceptional ring over $R_E$.  
\end{proof}

\begin{exm}
By gluing an exceptional ring to a quaternion ring, we see is indeed possible for the set in Corollary \ref{closedexcept} to be a proper subset of $\Spec R$: for example, for $k$ a field we can take $R=k[x,y]/(xy)$ and in ($U$) let $a=b=c=x$, $u'=v'=w'=y$, and $u=v=w=x+y$.
\end{exm}

\begin{exm}
Let $k$ be a field and let $R=k[\eps]/(\eps^2)$.  Consider the algebra $B$ with multiplication laws in ($U$) with $a=b=c=u'=v'=w'=\eps$ and $u=v=w=0$.  Then $B$ is neither a quaternion ring nor an exceptional ring; indeed, we have $\Spec R = \{(\eps)\}$ and $R_Q=R_E=R/(\eps)=k$.  
\end{exm}

\begin{lem} \label{ifreduced}
Suppose $R$ is reduced.  Then in the decomposition $\Spec R_Q \cup \Spec R_E \hookrightarrow \Spec R$ the map $R \to R_Q \times R_E$ is surjective.
\end{lem}

\begin{proof}
For each $\frakp \in \Spec R$, the ring $B/\frakp B$ is either quaternion ring or an exceptional ring (or both) by Remark \ref{uuvvwwdomain}.  Consequently, we have 
\[ \ker (R \to R_Q \times R_E) \subset \bigcap_{\frakp} \frakp = (0) \]  
so $R \to R_Q \to R_E$ is surjective.
\end{proof}

\begin{rmk}
Note that if $\Spec R_{QE}=\Spec R_Q \cap \Spec R_E$, then $B_{R_{QE}}$ is everywhere locally isomorphic to $R_\frakp[i,j,k]/(i,j,k)^2$ for $\frakp$ a prime of $R_{QE}$.
\end{rmk}

We continue in the direction of proving Theorem A with the following proposition.

\begin{prop} \label{exceptfree}
Let $B$ be an $R$-algebra of rank $4$ with a standard involution.  Then $B$ is an exceptional ring if and only if the canonical exterior form $\phi_B$ of $B$ is identically zero.
\end{prop}

\begin{proof}
If $B$ is an exceptional ring, then the canonical exterior form $\phi_B$ is zero since it is zero locally by Example \ref{canonexceptfree}.  

Conversely, suppose that the canonical exterior form $\phi_B$ is zero.  Since being an exceptional ring is a local property (Lemma \ref{exceptlocalold}), we may assume that $B$ is free over $R$ with multiplication laws ($U$).  Applying the standard involution, we have
\[ (u-i)(w-k)=\overline{i}\,\overline{k}=\overline{ki}=bj + w'\overline{i}=bj+w'(u-i) \]
so 
\[ ik=u(w'-w)-(w'-w)i+bj+uk. \]
and similarly we have the products $ji$ and $kj$.  

We again identify $\tbigwedge^4 B \xrightarrow{\sim} R$ by $1 \wedge i \wedge j \wedge k \mapsto -1$.  We compute that $1 \wedge i \wedge j \wedge ij = c=0$ and by symmetry $a=b=0$.  Similarly, we have
\[ \phi(i \wedge (j-k)) = - 1\wedge i \wedge k \wedge ij - 1 \wedge i \wedge j \wedge ik = -u'+u=0. \]
Thus $u'=u$, and by symmetry $v'=v$ and $w'=w$.  It follows then that $B=R \oplus M$ is an exceptional ring with $M=Ri\oplus Rj\oplus Rk$.
\end{proof}

For $x \in B$, let $\chi_L(x;T)$ (resp.\ $\chi_R(x;T)$) denote the characteristic polynomial of left (resp.\ right) multiplication by $x$ and $\Tr_L(x)$ (resp.\ $\Tr_R(x)$) denotes the trace of left (resp.\ right) multiplication by $x$.  These are defined locally given a choice of basis, but then are unique and so can be defined globally; see \cite[Lemma 1.4]{Voightlowrank}.  Let $\mu(x;T)=T^2-\trd(x)T+\nrd(x)$.

\begin{prop} \label{mondoequiv}
Let $B$ be a free $R$-algebra of rank $4$ with a standard involution.  Then the following are equivalent:
\begin{enumroman}
\item[(i)] $B$ is a quaternion ring;
\item[(ii${}_L$)] $\chi_L(x;T) = \mu(x;T)^2$ for all $x \in B$;
\item[(ii${}_R$)] $\chi_R(x;T) = \mu(x;T)^2$ for all $x \in B$;
\item[(iii${}_L$)] $2\trd(x) = \Tr_L(x)$ for all $x \in B$;
\item[(iii${}_R$)] $2\trd(x) = \Tr_R(x)$ for all $x \in B$;
\end{enumroman}
If $2$ is a nonzerodivisor in $R$, then these are further equivalent to
\begin{enumroman}
\item[(iv)] $\chi_L(x;T)=\chi_R(x;T)$ for all $x \in B$.
\end{enumroman}
\end{prop}

\begin{proof}
As in the proof of Proposition \ref{closedexcept}, we may assume that $B$ is free with a good basis, in which case $B$ is a quaternion ring if and only if $u'=v'=w'=0$.

Let $\xi=xi+yj+zk \in B \otimes_R R[x,y,z]$.  Then we compute that
\[ \mu(\xi;T)=T^2-(ux+vy+wz)T +n(x,y,z) \]
where
\[ -n(x,y,z)=bcx^2 + (uv-cw)xy + (uw-bv)xz + acy^2 + (vw-au)yz + abz^2, \]
and moreover
\begin{align*} 
\chi_L(\xi;T) &= \mu(\xi;T)(\mu(\xi;T) - (u'+v'+w')T) \\
\chi_R(\xi;T) &= \mu(\xi;T)(\mu(\xi;T) + (u'+v'+w')T).
\end{align*}
Each of the equivalences now follow easily.
\end{proof}

This general characterization allows us to distinguish quaternion rings and exceptional rings as follows.

We now put the pieces together and prove Theorem A.

\begin{proof}[Proof of Theorem A]
We combine Proposition \ref{mondoequiv}, Corollary \ref{closedexcept}, Lemma \ref{ifreduced}, and Proposition \ref{exceptfree}.  To prove the equivalence (iv), we 
 recall that the only free algebra which is both a quaternion ring and an exceptional ring is commutative, so to prove that $R=R_Q$ it suffices to note that $\Spec R_E \subset \Spec R_Q$ if and only if for every maximal ideal $\frakm \in \Spec R_E$ we have that $B \otimes R/\frakm$ is commutative.
\end{proof}

We now conclude this paper by relating our notion of quaternion ring to two common interpretations in the literature.  

First, we consider Azumaya algebras.  An $R$-algebra $B$ is \emph{Azumaya} if $B$ is central and \emph{$R$-simple} (or \emph{ideal}, as in Rao \cite{Rao}), that is to say every two-sided ideal $I$ of $B$ is of the form $\fraka B$ with $\fraka = I \cap R$, or equivalently that any $R$-algebra homomorphism $B \to B'$ is either the zero map or injective.  Equivalently, $B$ is Azumaya if and only if $B/\frakm B$ is a central simple algebra over the field $R/\frakm$ for all maximal ideals $\frakm$ of $R$, or if the map $B^e=B \otimes_R B^o \to \End_R B$ by $x \otimes y \mapsto (z \mapsto xzy)$ is an isomorphism, where $B^o$ is the opposite algebra.  (For a proof of these equivalences, see Auslander and Goldman \cite{AusGold} or Milne \cite[\S IV.1]{Milne}.)

Suppose that $B$ is an $R$-algebra of rank $4$ with a standard involution.  Then if $B$ is Azumaya then in particular $B$ is a quaternion ring.  A quaternion ring is Azumaya if and only if $D(B) = R$, or equivalently if the twisted isometry class of ternary quadratic modules associated to $B$ is semiregular (i.e. $D(M,I,q)=R$).

Next, we consider crossed products; we compare our notion of quaternion ring to that of Kanzaki \cite{Kanzaki0,Kanzaki}.  

Let $S$ be a quadratic $R$-algebra, let $J$ be an $S$-module which is projective of rank $2$ as an $R$-module and let $b:J \times J \to S$ be a \emph{Hermitian form} on $J$, an $S$-bilinear form which satisfies $\overline{b(x,y)}=b(y,x)$ for all $x,y \in J$, where $\overline{\phantom{x}}$ is the standard involution on $S$.  With this data, we define the $R$-algebra $\quat{S,J,b}{R}=S \oplus J$ with the multiplication rules $xy=b(x,y)$ for $x,y \in J$ and $xu=\overline{u}x$ for $u \in S$ and $x \in J$.  In particular, we note that $x^2 = b(x,x) \in R$ for all $x \in J$.  One can verify in a straightforward way that $\quat{S,J,b}{R}$ is indeed an (associative) $R$-algebra equipped with a standard involution given by $x \mapsto -x$ on $J$.

\begin{rmk}
If $B=S \oplus J$ is an $R$-algebra of rank $4$ with a standard involution such that $S$ is a quadratic $R$-algebra and $J$ an $S$-module, then it is not necessarily true that $J^2 \subset S$ as in the example above.  Indeed, the free $R$-algebra $B=R\oplus Ri \oplus Rj \oplus R ij$ with $S=R \oplus Ri$ and $J=Sj=Rj\oplus Rij$ subject to $i^2=i$, $j^2=j$, and $ij+ji=i+j$ has $j^2=j(ij)=j \not\in S$.
\end{rmk}

The $S$-module $J$ is \emph{traceable} if $S$ and $J$ give the same trace map $\Tr:S \to R$, i.e., if the trace of the $R$-linear endomorphism given by (left) multiplication on $S$ and on $J$ are equal.  An $R$-algebra $B$ which is of the form $B \cong \quat{S,J,b}{R}$ where $J$ is traceable is called a \emph{crossed product}.  

Every crossed product $B \cong \quat{S,J,b}{R}$ is a quaternion ring.  Indeed, by hypothesis, we have $\Tr(x)=2\trd(x)$ for all $x \in S$ and $\Tr(x)=0=2\trd(x)$ for all $x \in J$, so $B$ is a quaternion ring by Theorem A.  (This is also easy to see if $J$ is projective over $S$: then there exists $u \in S$ and $x \in J$ such that $x,ux$ are $R$-linearly independent, hence $\phi(u \wedge x)=1 \wedge u \wedge x \wedge ux \neq 0$, and making this argument locally we conclude by Proposition \ref{exceptfree} that $B=\quat{S,J,b}{R}$ is nowhere locally isomorphic to an exceptional ring, and in particular $B$ is a quaternion ring.

\begin{prop}
A quaternion ring $B$ is a crossed product if and only if it corresponds (in the bijection of Theorem B) to a ternary quadratic module $(M,I,q)$ such that $M$ can be written as an orthogonal direct sum $M \cong Q \oplus N$ where $Q,N$ are of ranks $1,2$.  
\end{prop}

\begin{proof}
We first compute the canonical exterior form $\phi$ corresponding to a crossed product $B=\quat{S,J,b}{R}$.  We find that the domain of $\phi$ is
\[ \tbigwedge^2(B/R) \cong \tbigwedge^2(S/R \oplus J) \cong (S/R \otimes_R J) \oplus \tbigwedge^2 J \]
and we claim that in fact the summands $S/R \otimes_R J$ and $\tbigwedge^2 J$ are orthogonal under $\phi$.  It is enough to prove this locally, in which case we can write $S=R \oplus Ri$ and $J=(R \oplus Ri)j=Rj \oplus Rk$ with $i,j \in B$ and $k=ij$.  Then $i \wedge j,k \wedge i,j \wedge k$ is a basis for $\tbigwedge^2(B/R)$, and we compute that
\[ \phi(j \wedge (k-i)) - \phi(i \wedge j) - \phi(j \wedge k) = 
1 \wedge j \wedge k \wedge ij - 1 \wedge j \wedge i \wedge jk = 0 \]
so that $i \wedge j$ is orthogonal to $k \wedge i$, and similarly that $i \wedge j$ is orthogonal to $j \wedge k$.  

Conversely, suppose that the ternary quadratic module $(M,I,q)$ has $M=Q \oplus N$ an orthogonal direct sum.  From the expression (\ref{whoocliffnelly}) we compute that that
\[ C^0(M,I,q) \cong \frac{R \oplus (N \otimes N \otimes I\spcheck)}{J^0(q|_N)} \oplus 
(N \otimes Q)=S \oplus J \]
as $R$-modules, where $J^0(q|N)$ is the $R$-module generated by elements of the form 
\[ x \otimes x \otimes f - 1 \otimes f(q(x)) \] 
for $x \in N$ and $f \in I\spcheck$.  To verify that $J$ is an invertible $S$-module and that the multiplication map on $J$ gives a Hermitian form $b:J \times J \to S$, it suffices to prove this locally, whence we may assume that $N=Re_1 \oplus Re_2$ and $Q=Re_3$ as in Example \ref{Qfree}, so that
\[ q(xe_1+ye_2+ze_3)=q(x,y,z)=ax^2+by^2+cz^2+wxy \]
where by orthogonality we have $u=v=0$.  Then in the above we have $1,k$ a basis for $S$ and $i,j$ a basis for $J$ as an $R$-module.  Thus, the multiplication laws ($Q$) read
\begin{align*}
i^2 &= -bc & jk &= a\overline{i}=-ai & kj &= ai + wj \\
j^2 &= -ac & ki &= b\overline{j}=-bj & ik &= wi + bj \\
k^2 &= wk-ab & ij &= c\overline{k}=cw-ck & ji &= ck.
\end{align*}
This verifies that $SJ \subset J$, so $J$ is an $S$-module, that $J$ is traceable, and that the multiplication map on $J$ yields a Hermitian form on $J$, e.g.\ $ik=\overline{k}i$ and $jk=\overline{k}j$, and $ij=c\overline{k}=\overline{ck}=\overline{ji}$.
\end{proof}

\begin{rmk}
Wood \cite{MWood}[Chapter 2] has shown that there is a bijection between isometry classes of binary quadratic modules and isomorphism classes of pairs $(S,J)$ where $S$ is a quadratic $R$-algebra and $J$ is a traceable $S$-module.  

There is a natural map from binary quadratic modules to ternary quadratic modules, where we associated to a binary quadratic module $q:N \to I$ the ternary quadratic module $q \oplus \la 1 \ra:N \oplus R \to I$ where $(q \oplus \la 1 \ra)(x,r) = q(x) + r^2$ for $x \in N$ and $r \in R$.  This associated ternary quadratic module then gives rise in the correspondence of Theorem B to the crossed product $B \cong \quat{S,J,b}{R}$ and the parity factorization $(\bigwedge^2 N \otimes (I\spcheck)^{\otimes 2})^{\otimes 2} \otimes I \xrightarrow{\sim} \bigwedge^4 B$, where $(S,J)$ is as associated by Wood.  The Hermitian form $b:J \times J \to S$ is defined locally as above with $c=1$.  In these associations, the isometry class of the binary quadratic module determines that of the ternary quadratic module and correspondingly the isomorphism classes of the traceable module as well as the quaternion ring (equipped with its specified parity factorization).

This correspondence was worked out partially by Wood (private communication).
\end{rmk}

\end{document}